\documentclass[11pt,letterpaper,reqno]{amsart}
\usepackage{amssymb,amsmath,amsthm,amsfonts}
\usepackage{tikz}

\newtheorem{theorem}{Theorem}
\newtheorem{lemma}[theorem]{Lemma}

\newcommand{\R}{\mathbb{R}}
\newcommand{\Z}{\mathbb{Z}}

\renewcommand{\S}{\mathcal{S}}

\numberwithin{equation}{section}

\begin{document}

\title{An $\mathrm{L}^4$ estimate for a singular entangled quadrilinear form}
\author{Polona Durcik}
\address{Polona Durcik, Universit\"at Bonn, Endenicher Allee 60, 53115 Bonn, Germany}
\email{durcik@math.uni-bonn.de}
\date{\today}

\subjclass[2010]{Primary 42B15; Secondary 42B20.}
%\keywords{multilinear form, twisted paraproduct}

\begin{abstract} 
The twisted paraproduct can be viewed as a two-dimensional trilinear form which appeared in the work by Demeter and Thiele on the two-dimensional bilinear Hilbert transform. 
$\mathrm{L}^p$ boundedness of the twisted paraproduct is due to Kova\v{c}, who in parallel established estimates for the dyadic model of a closely related quadrilinear form. 
We prove an  $(\mathrm{L}^4, \mathrm{L}^4, \mathrm{L}^4, \mathrm{L}^4)$ bound for the continuous model of the latter by adapting the 
technique of Kova\v{c} to the continuous setting. 
The mentioned forms belong to a larger class of operators with general modulation invariance. 
Another  instance of such is the triangular Hilbert transform, which controls issues related to two commuting transformations in ergodic theory, and for which $\mathrm{L}^p$ bounds remain an open problem. 
\end{abstract}

\maketitle

%%%%%%%%%%%%%%%%%%%%%%%%%%%%%%%%%%%%%%%%%%%%%%%%%%%%%%%%%%%%%%%%%%%%%%%%%%%%%%%%%%%%%%%%%%%%%%%%%%%%%%%%%%%%%%%%%

\vspace{-0.35cm}
\section{Introduction}
For four functions $F_1,F_2,F_3,F_4$ on $\R^2$ we denote their "entangled product"
\begin{align}\label{f}
\boldsymbol{F}_{(F_1,F_2,F_3,F_4)}(x,x',y,y'):=
F_1(x,y)F_2(x',y)F_3(x',y')F_4(x,y').
\end{align}
Let $m$ be a bounded  function on $\R^2$, smooth away from the origin and  
satisfying\footnote{For two non-negative quantities $A$ and $B$ we write $A\lesssim B$ if there is an absolute constant $C>0$ such that $A\leq C B$. 
We write $A\lesssim_P B$ if the constant depends on a set of parameters $P$.}
\begin{align}\label{symest}
 |\partial^\alpha m(\xi,\eta)|\lesssim (|\xi|+|\eta|)^{-|\alpha|}
\end{align} for all multi-indices $\alpha$ up to some large finite order. 
With any such $m$ we associate a quadrilinear form $\Lambda=\Lambda_m$ defined as\footnote{The Fourier transform we use is defined in \eqref{ft-def}.} 
$$\Lambda(F_1,F_2,F_3,F_4):=  \int_{\R^2}  \widehat{\boldsymbol{F}}(\xi,-\xi,\eta,-\eta)m(\xi,\eta) d\xi d\eta$$
for Schwartz functions $F_j\in \S(\R^2)$,   where $\boldsymbol{F}:=\boldsymbol{F}_{(F_1,F_2,F_3,F_4)}$. The object of this   paper is to establish the following bound.
\begin{theorem} \label{mainthm}
The quadrilinear form $\Lambda$ satisfies the estimate
\begin{align}\label{mainest}
|\Lambda(F_1,F_2,F_3,F_4)| \lesssim 
\| F_1 \|_{\mathrm{L}^4(\R^2)}\| F_2 \|_{\mathrm{L}^4(\R^2)}\| F_3 \|_{\mathrm{L}^4(\R^2)}\| F_4 \|_{\mathrm{L}^4(\R^2)}.
\end{align}
\end{theorem}
When $m$ is identically one,  $\Lambda$ corresponds to the pointwise product form
\begin{align*}
\Lambda(F_1,F_2,F_3,F_4) = 
%(F_1,F_2,F_3,F_4) \mapsto 
 \int_{\R^2} F_1(x,y)F_2(x,y)F_3(x,y)F_4(x,y)dxdy.
\end{align*}
The bound  \eqref{mainest} is  then an immediate consequence of H\"older's inequality and holds in a larger range of exponents. In general,  we can formally write   $\Lambda(F_1,F_2,F_3,F_4)$ as
\begin{align}\label{lambda-kernel}
 \int_{\R^4} F_1(x,y)F_2(x',y)F_3(x',y')F_4(x,y')\kappa(x'-x,y'-y)dxdx'dydy',
\end{align}
where $\kappa$ is a two-dimensional Calder\'on-Zygmund kernel. 

The motivation for these objects originates in the  study of the {twisted paraproduct} \cite{kovac:tp}. We call the {\em twisted paraproduct} a trilinear form $T=T_m$ defined as
$$T(F_1,F_2,F_3):=\Lambda(F_1,F_2,F_3,1).$$
That is, the fourth function in the entangled product $\boldsymbol{F}$ is  the constant function one. 
The form $T$ was proposed by Demeter and Thiele \cite{demTh:BHT}  as the dual of a particular case of the two-dimensional  bilinear Hilbert transform. 
This was the only  case which could not be treated with the  time-frequency techniques in \cite{demTh:BHT}. 
Lack of applicability of the latter is closely related with general modulation symmetries that the operators $T$ and $\Lambda$ exhibit. 
An example of such a symmetry is that for any $g\in \mathrm{L}^\infty(\R)$ we have  invariance
$$\Lambda((1\otimes g) F_1, F_2,F_3,F_4)=\Lambda(F_1,(1\otimes g)F_2,F_3,F_4),$$
where $(f\otimes g)(x,y):=f(x)g(y)$. This is evident from their entangled structure. One can informally say that the generalized modulation invariance is present since
 several functions depend on the same one-dimensional variable. 

First bounds for $T$ are due to Kova\v{c} \cite{kovac:tp}, who established
\begin{align}\label{boundTP}
|T(F_1,F_2,F_3)|\lesssim_{(p_j)} \|F_1\|_{\mathrm{L}^{p_1}(\R^2)}\|F_2\|_{\mathrm{L}^{p_2}(\R^2)}\|F_3\|_{\mathrm{L}^{p_3}(\R^2)}
\end{align}
whenever 
${1}/{p_1} + {1}/{p_2} + {1}/{p_3} =1$ and $2<p_1,p_2,p_3<\infty$. His approach relied on the Bellman function technique. The fiber-wise  Calder\'on-Zygmund decomposition of Bernicot \cite{bernicot:fw}
extended the range of exponents to $1<p_1,p_3 < \infty$,  $2<p_2\leq \infty$.

Kova\v{c} observed that adding the fourth function $F_4$ to $T$ completes the cyclic structure of the form and results in an object with a high degree of symmetry. 
For instance, for even kernels $\kappa$ one has $\Lambda(F_1,F_2,F_3,F_4)=\Lambda(F_3,F_4,F_1,F_2)$.
Moreover, $T$ and $\Lambda$ can be seen as the smallest non-trivial examples of a  family of entangled multilinear forms associated with   bipartite graphs, whose dyadic models were studied in \cite{kovac:bf}.

To prove \eqref{boundTP}, Kova\v{c} passed through a dyadic version of $\Lambda$, which we call $\Lambda_\mathrm{d}$.  
He considered \eqref{lambda-kernel} with $\kappa$ replaced by the perfect (dyadic) Calder\'on-Zygmund kernel 
\begin{align}\label{kernel:dyad}
\sum_{I\times J} \varphi^\mathrm{d}_I(x)\varphi^\mathrm{d}_I(x')\psi^\mathrm{d}_J(y)\psi^\mathrm{d}_J(y').
\end{align}
The sum in \eqref{kernel:dyad} runs over all dyadic squares\footnote{A dyadic square 
is a product of two dyadic intervals of the same length. A dyadic interval is an interval of the form $[2^km,2^k(m+1)),\, k,\,m \in \Z$.} 
$I\times J$ in $\R^2$. For a dyadic interval $I$, 
the scaling function and the Haar function are   defined as\footnote{We write $\mathbf{1}_A$ for the characteristic function of a set $A\subseteq \R$.}
\begin{align*}
\varphi^\mathrm{d}_{I}:=|I|^{-1/2}\mathbf{1}_I\;\; \mathrm{and}\;\;
\psi^\mathrm{d}_{I}:=|I|^{-1/2}\left ( \mathbf{1}_{I_\mathrm{left\, half}}-\mathbf{1}_{I_\mathrm{right\, half}} \right ),
\end{align*}
respectively.
The large range of exponents  in \eqref{boundTP}  was achieved by first proving a  local bound for the variant of $\Lambda_\mathrm{d}$ with the summation in \eqref{kernel:dyad} running over a subset of dyadic squares called trees. Then, $F_4$ was set equal to $1$ and
  contributions of a single tree were   integrated into a global estimate. This established the desired estimate for the 
  dyadic model $T_\mathrm{d}$ of $T$ defined by the relation   $T_\mathrm{d}(F_1,F_2,F_3):= \Lambda_\mathrm{d}(F_1,F_2,F_3,1)$.

It remained to tackle $T$ for continuous kernels.  Via the cone decomposition, see  \cite{th:wpa}, this problem was first reduced to the  case  
\begin{align}\label{kernel}
\kappa(s,t)=\sum_{k\in \Z}2^k\varphi(2^ks)2^k\psi(2^kt),
\end{align}
where $\varphi,\psi \in \S(\R) $ are two Schwartz functions and $\widehat\psi$ is supported on $\{1\leq |\xi| \leq 2 \}$. 
Their dilations  by $2^k$ can be seen as continuous analogues of $\varphi_I$, $\psi_I$. 
In \cite{kovac:tp}, the bound \eqref{boundTP} was  finally established by relating the special case of $T$, associated with \eqref{kernel}, to the dyadic $T_\mathrm{d}$. 
This was done by rewriting $T$ and $T_\mathrm{d}$ using convolutions and martingale averages in the respective cases. Then, the  square functions of Jones, Seeger and Wright \cite{jsw} were used to compare the continuous with the discrete averaging operator.

A natural question  is what we can say about $\Lambda$    if the function $1$ is replaced by  a completely general function $F_4$. For the dyadic model $\Lambda_\mathrm{d}$,  Kova\v{c}  proved   $(\mathrm{L}^{p_1},\mathrm{L}^{p_2},\mathrm{L}^{p_3},\mathrm{L}^{p_4})$ estimates whenever $p_j$ are H\"older-type  exponents satisfying $2< p_j< \infty$ for all $j$. See \cite{kovac:tp} and \cite{kovac:bf}. 
However, due 
to the more complex structure of the form, 
one cannot efficiently rewrite $\Lambda$ and $\Lambda_\mathrm{d}$ in a similar way as $T$ and $T_\mathrm{d}$ to  exploit the mentioned square functions. Thus,  the question about $\Lambda$ associated with any $m$ satisfying \eqref{symest} remains.

In the present note we obtain an answer in this direction by adapting the technique used to treat $\Lambda_\mathrm{d}$ in \cite{kovac:tp} to the continuous setting. 
We address  the simplest $\mathrm{L}^4$ case only. It is expected that suitable tree decompositions will eventually enable us  
to prove \eqref{mainest} for a larger range of exponents.  However, for the considered quadrilinear form we cannot make use of the fiber-wise Calder\'on-Zygmund decomposition by Bernicot.

The core argument in \cite{kovac:tp} intertwines two applications of the Cauchy-Schwarz inequality,  
 which gradually separates the functions $F_j$, 
and two applications of an algebraic identity, which "interchanges" the functions $\varphi^\mathrm{d}$ and $\psi^\mathrm{d}$. 
This identity, involving a telescoping argument in the dyadic case, is now replaced by a differential equality combining the fundamental theorem of calculus and the Leibniz rule.  
The main issue in the continuous setup is that the mentioned algebraic trick can be applied twice if the functions $\varphi,\psi$, decomposing the kernel, are sufficiently symmetric. For example, even functions would work. Moreover, they need to possess enough decay and have certain smoothness properties,  which should be maintained throughout the process. Suitable candidates  which fulfil the requirements are, for instance, the Gaussian exponential functions. 

Although we cannot expect our functions $\varphi,\psi$ to be even, much less the Gaussian exponential functions,   we are able to overcome the mentioned restrictions 
as follows. First, the reduction to the case of a  concrete kernel, such as  \eqref{kernel}, is done by a careful choice of the functions $\varphi,\psi$. 
This way we obtain some of the required symmetry and regularity. Second, after each application of the Cauchy-Schwarz inequality we dominate certain
  functions   with a suitable superposition of dilated  Gaussian exponential functions. This gradually reduces the two algebraic steps to the case of Gaussians, 
which most resembles the dyadic telescoping trick.

Besides extending the exponent range, it would be of interest to obtain boundedness results for the continuous models of the forms from \cite{kovac:bf}, associated with bipartite graphs.   

Let us  briefly comment on another related open problem. There is a question of establishing $\mathrm{L}^p$ estimates for the akin trilinear form
\begin{align*}
\Lambda_\bigtriangleup (F_1,F_2,F_3) := \int_\R \widehat{\boldsymbol{F}}(\xi,\xi,\xi)\mathrm{sgn}(\xi) d\xi
\end{align*}
where the entangled product $\boldsymbol{F}$ is now given by
\begin{align*}
\boldsymbol{F}(x,y,z):=F_1(x,y)F_2(y,z)F_3(z,x).
\end{align*}
Passing to the spatial side, one has up to a constant
\begin{align*}
\Lambda_\bigtriangleup(F_1,F_2,F_3) = \int_{\R^3}F_1(x,y)F_2(y,z)F_3(z,x)\frac{-1}{x+y+z}dxdydz.
\end{align*}
The  structure of $\Lambda_\bigtriangleup$ corresponds to the three-cycle and  for this reason it is   called the {\em  triangular Hilbert transform}. 
No $\mathrm{L}^p$ bounds for $\Lambda_\bigtriangleup$ or for its dyadic model are known. Lack of the bipartite structure prevents to approach it with the techniques from \cite{kovac:bf}. 

Boundedness of $\Lambda_\bigtriangleup$ would  imply boundedness for certain instances of the two- dimensional bilinear Hilbert transform and the twisted paraproduct.  
Further interest  in $\Lambda_\bigtriangleup$  arises from ergodic theory. 
It is proposed by Demeter and Thiele \cite{demTh:BHT} to approach the open question of
pointwise almost everywhere convergence for ergodic averages
\begin{align*}
 \frac{1}{N} \sum_{n=1}^N f(T^nx)g(S^nx),
\end{align*}
where $S,\,T:X\rightarrow X$ are two commuting measure preserving transformations on a probability space $X$, via an examination of the triangular Hilbert transform.\\

{\bf Acknowledgement.}
I am grateful to my  advisor Prof. Christoph Thiele for his constant support, valuable consultations on the problem and numerous suggestions on improving the text. I am thankful to Vjekoslav Kova\v{c} for useful discussions.

%%%%%%%%%%%%%%%%%%%%%%%%%%%%%%%%%%%%%%%%%%%%%%%%%%%%%%%%%%%%%%%%%%%%%%%%%%%%%%%%%%%%%%%%%%%%%%%%%%%%%%%%%%%%%%%%%

\section{Decomposition of the symbol}
\label{sec:cone}
To begin, we  reduce the general symbol to a particular function by decomposing $m$ into pieces which are supported on certain subsets of two double cones. We follow the main ideas discussed  in \cite{th:wpa}. However, we do not discretize, but rather keep continuum in the scale. 

The Fourier transform we shall use throughout this note is defined as 
\begin{align}\label{ft-def}
\widehat{f}(\omega):=\int_{\R^n} f(\tau)e^{-2\pi i \tau\cdot \omega } d\tau.
\end{align}

By a smooth partition of unity and symmetry    in  $\xi,\eta$   we may assume that $m$ is supported on 
the double cone   
$$\{(\xi,\eta) : |\xi|\leq 1.001 |\eta|\}$$
centered around the $\eta$-axis.
Choosing double cones over single cones will allow us to use  functions that are symmetric around the origin. 
We can choose the partition of unity 
 such that \eqref{symest} is preserved, possibly with a different constant.
 \begin{figure}[htb] 
\centering
\begin{tikzpicture}[scale=0.6,rotate=90]
%\draw [help lines] (-3,-3) grid (3,3); 
\fill [fill=gray, opacity = 0.1] (-3.5,-3.5) -- (-3.5/1.01,-3.5)-- (0,0) -- (-3.5/1.01,3.5) -- (-3.5,3.5);
\fill [fill=gray, opacity = 0.1] (3.5,-3.5)--(3.5/1.01,-3.5) -- (0,0) -- (3.5/1.01,3.5) --(3.5,3.5);
%/upper slice
\draw [thick,domain=-atan(1.01):atan(1.01)] plot ({1.7*cos(\x)}, {1.7*sin(\x)});
\draw [thick,domain=-atan(1.01):atan(1.01)] plot ({2.7*cos(\x)}, {2.7*sin(\x)});
%/lower slice
\draw [thick,domain=-atan(1.01)+180:atan(1.01)+180] plot ({1.7*cos(\x)}, {1.7*sin(\x)});
\draw [thick,domain=-atan(1.01)+180:atan(1.01)+180] plot ({2.7*cos(\x)}, {2.7*sin(\x)});
%draw lines
\draw[dashed] [domain=-3.5/1.01:3.5/1.01, samples=50] plot (\x, 1.001*\x);
\draw[dashed] [domain=-3.5/1.01:3.5/1.01, samples=50] plot (\x, -1.001*\x);
\draw[->] (-3.5,0) -- (3.5,0);
\draw[<-] (0,-3.5) -- (0,3.5);
\draw[right, font=\footnotesize] (3.25,0) node {$\eta$};
\draw[below, font=\footnotesize] (0,-3.4) node {$\xi$};
%draw rectangles
%/upper rectangle
\draw (1,-2) -- (1,2) -- (3,2) -- (3,-2) -- cycle;
%/lower rectangle
\draw (-1,-2) -- (-1,2) -- (-3,2) -- (-3,-2) -- cycle;
\draw[fill=white] (0,0) circle (1mm);
\end{tikzpicture}
\caption{Decomposition of $m$.} \label{fig:cone}
\end{figure}
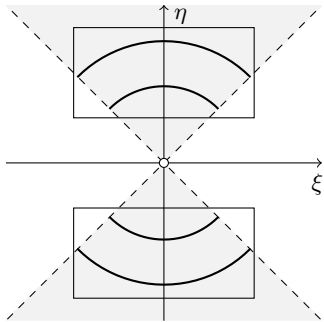

By $B(0,R)$ we denote the ball of radius $R$  centered at the origin in $\R^2$. Let $\theta$ be a function on $\R^2$ such that $\widehat{\theta}$ is smooth, real, radial and supported in the annulus $B(0,2.7)\setminus B(0,1.7)$. 
We normalize so that for every $(\xi,\eta)\neq 0$ we have
$$\int_0^\infty\widehat{\theta}(t\xi, t\eta)\frac{dt}{t}=1.$$ This can be achieved, since $\widehat{\theta}$ is radial and  supported away from $0$. Then we can write
$$m(\xi,\eta)=\int_0^\infty m_t(\xi,\eta) \frac{dt}{t},$$
where $m_t(\xi,\eta):=m(\xi,\eta)\widehat{\theta}(t\xi,t\eta)$.

In what follows we will be working with certain smooth bump functions, for which we need the following technical lemma. Its proof can be found in the appendix.
\begin{lemma}\label{lemma:bump}
Let $\varepsilon:=0.001$.
There exists a non-negative real-valued function $f\in C_0^\infty(\R)$ which is supported in $[1,3]$,  even about $2$  and  constantly equal $1$ on $[1+\varepsilon,3-\varepsilon]$, such that
$f^{1/2}$ and
\begin{align}\label{square-root}
\Big (\int_x^\infty \frac{f(t)+f(-t)}{t} dt \Big )^{1/2}
\end{align} belong to $C_0^\infty(\R)$.
\end{lemma}

Now consider $m_1$. Its support is contained in the union of the rectangles
\begin{align*}
[-2,2] \times [-3,-1] \hspace{1cm}\mathrm{and}\hspace{1cm} [-2,2] \times [1,3].
\end{align*}
Let $f$ be the function from Lemma \ref{lemma:bump} and let $\vartheta_1,\vartheta_2 \in \S(\R)$ be such that  $\widehat{\vartheta_1}(\xi)=f((\xi+4)/2)$ and $\widehat{\vartheta_2}(\xi)=f(\xi)+f(-\xi)$. 
Then $\widehat{\vartheta_1}\otimes\widehat{\vartheta_2}$  equals $1$ on the support of $m_1$.
Thus, by dilating $\widehat{\vartheta_1},\widehat{\vartheta_2}$  in $t$,    for every $t>0$ we  can write
$$m_t(\xi,\eta)=m_t(\xi,\eta)  \widehat{\vartheta_1}(t\xi)\widehat{\vartheta_2}(t\eta). $$ 
This can be  rewritten further using the Fourier inversion formula on $m_t$ as 
\begin{align*}
m_t(\xi,\eta)=\left( \int_{\R^2}\mu_t(u,v)  e^{2\pi i ut\xi}e^{2\pi i vt\eta} dudv \right ) \widehat{\vartheta_1}(t\xi) \widehat{\vartheta_2}(t\eta),
\end{align*} 
where $\mu_t :=t^2\widehat{m_t}(t\cdot,t\cdot)$. Integrating by parts sufficiently many times, using that \eqref{symest} holds for $m(\xi/t,\eta/t)$ uniformly in $t$ and considering the support of $m_t$ we obtain
\begin{align*}%\label{coefest} \nonumber
|\mu_t(u,v)| & = t^2 \left| \int_{\R^2} m_t(\xi,\eta)e^{-2\pi i  (u t\xi + v t\eta)} d\xi d\eta\right | = \left| \int_{\R^2} m_t \Big (\frac{\xi}{t},\frac{\eta}{t} \Big )e^{-2\pi i (u \xi + v \eta)} d\xi d\eta\right | \\
& \lesssim (1+|u|)^{-12}(1+|v|)^{-12}.
\end{align*} 

Define $\varphi^{(u)},\psi^{(v)}$ by
\begin{align}\label{root-1}
 \widehat{\varphi^{(u)}}(\xi)&:=(1+|u|)^{-5}(\widehat{\vartheta_1}(\xi))^{1/2}e^{\pi i u\xi},\\
\nonumber \widehat{\psi^{(v)}}(\eta)&:=(\widehat{\vartheta_2}(\eta))^{1/2}e^{\pi i v\eta}.
\end{align}
By Lemma \ref{lemma:bump} we have $(\widehat{\vartheta_1})^{1/2}\in C_0^\infty(\R)$, so the function $\varphi^{(u)}$ satisfies the bound 
\begin{align}\label{fct-est}
 |\varphi^{(u)}(x)|\lesssim (1+|x|)^{-5}
\end{align}
uniformly  in $u$. We will apply this fact in the following section.   
Now we can write
$$m(\xi,\eta)=\int_0^\infty  \int_{\R^2}  \widetilde{\mu}_t(u,v) (\widehat{\varphi^{(u)}}(t\xi))^2 (\widehat{\psi^{(v)}}(t\eta))^2 dudv \frac{dt}{t},$$
where the coefficients $\widetilde{\mu}_t$ are defined as $$\widetilde{\mu}_t(u,v):= (1+|u|)^{10}\mu_t(u,v).$$  
Note that $(\widehat{\vartheta_1})^{1/2}$ and $(\widehat{\vartheta_2})^{1/2}$ are real-valued and even, 
so $\varphi^{(u)},\psi^{(v)}$ are multiples of translates of real-valued functions and thus real-valued. 

To summarize, on a double cone we have decomposed $\Lambda(F_1,F_2,F_3,F_4)$ into
\begin{align*}
 \int_{\R^2} \int_0^\infty \widetilde{\mu}_t(u,v)  \int_{\R^2}  \widehat{\boldsymbol{F}}(\xi,-\xi,\eta,-\eta)(\widehat{\varphi^{(u)}}(t\xi))^2 (\widehat{\psi^{(v)}}(t\eta))^2 d\xi d\eta\frac{dt}{t} dudv.
\end{align*}
By the rapid decay of the coefficients $\widetilde{\mu}_t$ it will suffice to prove \eqref{mainest} for the form
\begin{align}\label{def:lambda}
 \int_0^\infty \Big | \int_{\R^2}  \widehat{\boldsymbol{F}}(\xi,-\xi,\eta,-\eta)(\widehat{\varphi^{(u)}}(t\xi ))^2 (\widehat{\psi^{(v)}}(t\eta))^2 d\xi d\eta \Big | \frac{dt}{t},
\end{align}
provided that the estimate holds uniformly in the parameters $u,v$. 

From now on we assume that the functions $F_j\in \S(\R^2)$ are real-valued, as otherwise we can split them into real and imaginary parts and use quadrisublinearity of \eqref{def:lambda}.

%%%%%%%%%%%%%%%%%%%%%%%%%%%%%%%%%%%%%%%%%%%%%%%%%%%%%%%%%%%%%%%%%%%%%%%%%%%%%%%%%%%%%%%%%%%%%%%%%%%%%%%%%%%%%%%%%

\section{Proof of Theorem \ref{mainthm}}
\label{sec:pf}
The proof proceeds with studying the special case \eqref{def:lambda}. For $t>0$ and   four functions $\phi_i\in \S(\R)$ we define  
\begin{align*}
L_{\phi_1,\phi_2,\phi_3,\phi_4}^t(F_1,F_2,F_3,F_4) :=  \int_{\R^2}  \widehat{\boldsymbol{F}}(\xi,-\xi,\eta,-\eta)\widehat{\phi_{1}} (t\xi)\widehat{\phi_{2}}  (-t\xi)\widehat{\phi_{3}} (t\eta)\widehat{\phi_{4}} (-t\eta)d\xi d\eta  .
\end{align*}
For the rest of this note we will consider objects of the type
\begin{align*}
\Lambda_{\phi_1,\phi_2,\phi_3,\phi_4}(F_1,F_2,F_3,F_4) := \int_0^\infty L^t_{\phi_1,\phi_2,\phi_3,\phi_4}(F_1,F_2,F_3,F_4) \frac{dt}{t}
\end{align*}
and 
\begin{align}\label{lambdaT-L}
\widetilde{\Lambda}_{\phi_1,\phi_2,\phi_3,\phi_4}(F_1,F_2,F_3,F_4) := \int_0^\infty \Big | L^t_{\phi_1,\phi_2,\phi_3,\phi_4}(F_1,F_2,F_3,F_4) \Big | \frac{dt}{t}.
\end{align}
Observe that   \eqref{def:lambda}  is obtained from \eqref{lambdaT-L} by choosing 
\begin{align*}
\begin{array}{ll}
\phi_1=\varphi^{(u)}, & \phi_3=\psi^{(v)},\\
\phi_2=\varphi^{(-u)}, & \phi_4=\psi^{(-v)}.
\end{array}
\end{align*} This follows from \eqref{root-1} and from the functions $\widehat{\vartheta_1},\widehat{\vartheta}_2$ being even. %We  denote the form \eqref{def:lambda} by $\widetilde{\Lambda}_{\varphi^{(u)},\varphi^{(-u)},\psi^{(v)},\psi^{(-v)}}$.

We  shall now express $L^t_{\phi_1,\phi_2,\phi_3,\phi_4}$ on the spatial side.
Let us denote by $[f]_t$ the $\mathrm{L}^1$-dilation of a function $f$ by a parameter $t>0$, i.e. $[f]_t(x):=t^{-1}f(t^{-1}x)$.  Then, $\widehat{[f]_t}(\xi)=\widehat{f}(t\xi)$.
Since the integral of the Fourier transform of a Schwartz function in $\R^4$  over the hyperplane 
$$\{(\xi,-\xi,\eta,-\eta) : \xi,\eta\in \R\}$$
  equals the integral of the  function itself over the perpendicular hyperplane
$$\{(p,p,q,q) : p,q\in \R\},$$
we can write $L^t_{\phi_1,\phi_2,\phi_3,\phi_4}(F_1,F_2,F_3,F_4)$ as
\begin{align*}
 \int_{\R^2} \boldsymbol{F}*([\phi_{1}]_t\otimes [\phi_{2}]_t \otimes [\phi_{3}]_t \otimes [\phi_{4}]_t)(p,p,q,q)
dp dq.
\end{align*}
Expanding the convolution, the last display can be identified as
\begin{align*}
   \int_{\R^6}  &  F_1(x,y)F_2(x',y)F_3(x',y')F_4(x,y')\\
&[\phi_{1}]_t(p-x)[\phi_{2}]_t(p-x')[\phi_{3}]_t(q-y)[\phi_{4}]_t(q-y') dxdx'dydy' dpdq.
\end{align*}

Now we are ready to start. The inequality \eqref{mainest}, which we want to establish, is homogeneous, so we may normalize 
$$\|F_j\|_{\mathrm{L}^4(\R^2)}=1,$$ 
for $j=1,2,3,4$.
Thus, we are set to show 
$$\widetilde{\Lambda}_{\varphi^{(u)},\varphi^{(-u)},\psi^{(v)},\psi^{(-v)}}(F_1,F_2,F_3,F_4)\lesssim 1.$$

The proof starts with an application of the Cauchy-Schwarz inequality. To preserve the mean zero property of $\psi^{(v)},\psi^{(-v)}$ we separate 
the involved functions according to the variables $y,y'$ and estimate $\widetilde{\Lambda}_{\varphi^{(u)},\varphi^{(-u)},\psi^{(v)},\psi^{(-v)}}(F_1,F_2,F_3,F_4)$ by 
\begin{align*}
 \int_0^\infty \int_{\R^4} \left | \int_{\R}F_1(x,y)F_2(x',y)[\psi^{(v)}]_{t}(q-y)  dy \right| \left|\int_{\R} F_3(x',y')F_4(x,y') [\psi^{(-v)}]_{t}(q-y') dy' \right| & \\
   [|\varphi^{(u)}|]_t(p-x) [|\varphi^{(-u)}|]_t(p-x') dxdx'dpdq\frac{dt}{t}. &
\end{align*}
Applying the Cauchy-Schwarz inequality  bounds this expression by the  product
\begin{align*}%\label{pf-cs}
{\Lambda}_{|\varphi^{(u)}|,|\varphi^{(-u)}|,\psi^{(v)},\psi^{(v)}}(F_1,F_2,F_2,F_1)^{1/2} 
 {\Lambda}_{|\varphi^{(u)}|,|\varphi^{(-u)}|,\psi^{(-v)},\psi^{(-v)}}(F_4,F_3,F_3,F_4)^{1/2}.
\end{align*}
We estimate the first factor of the above display, the second is dealt with similarly. 

To further separate the involved functions  we would like to apply the Cauchy-Schwarz inequality again, which now needs to be done in the complementary variables. 
So we need to "switch" the functions $\varphi^{(u)}$ and $\psi^{(v)}$. This is where we make use of the following lemma, a continuous analogue of the telescoping identity from \cite{kovac:tp}.
\begin{lemma} \label{tel}
Assume that we have two pairs of real-valued Schwartz functions $(\rho_i,\sigma_i)$, $i=1,2$,  which satisfy
\begin{align}\label{tel-theid}
-t\partial_t|\widehat{\rho_{i}}(t\tau)|^2=|\widehat{\sigma_{i}}(t\tau)|^2\hspace{0.5cm}\mathrm{for} \hspace{0.5cm} i=1,2.
\end{align}
Then with $c:=|\widehat{\rho_1}(0)|^2|\widehat{\rho_2}(0)|^2$ we have
\begin{align}\label{tel-id}
{\Lambda}_{\sigma_1,\rho_2} (F_1,F_2,F_3,F_4) + {\Lambda}_{\rho_1,{\sigma_2}}(F_1,F_2,F_3,F_4)
 =  c \int_{\R^2}F_1F_2F_3F_4,
\end{align}
where we have  denoted $\Lambda_{\sigma,\rho}= \Lambda_{\sigma,\sigma,\rho,\rho}$. 
\end{lemma}
\begin{proof}
By the fundamental theorem of calculus, 
\begin{align}\label{tel-ftc}
\int_0^\infty \partial_t (| \widehat{\rho_1}( t\xi)|^2
|\widehat{\rho_2}(t\eta)|^2) dt = -|\widehat{\rho_1} (0)|^2|\widehat{\rho_2}(0)|^2.
\end{align}
The left hand-side of \eqref{tel-ftc} equals
\begin{align}\label{tel:second-last}
 & \int_0^\infty t \partial_t  (|\widehat{\rho_1} (t\xi)|^2)
|\widehat{\rho_2}(t\eta)|^2 \frac{dt}{t}\\ \nonumber
+ & \int_0^\infty  |\widehat{\rho_1}(t\xi)|^2
t\partial_t  (|\widehat{\rho_2}(t\eta)|^2)  \frac{dt}{t}.
\end{align}
The functions $\rho,\sigma$ are real-valued, so $\overline{\widehat{\rho}}(\eta)=\widehat{\rho}(-\eta)$, and analogously for $\sigma$.
Together with \eqref{tel-theid} this shows that  \eqref{tel:second-last} can be written as
\begin{align}\label{pf-tel}
-\int_0^\infty \widehat{\sigma_1}(t\xi)\widehat{\sigma_1}(-t\xi)
\widehat{\rho_2}(t\eta)\widehat{{\rho_2}}(-t\eta) \frac{dt}{t}
-\int_0^\infty \widehat{\rho_1}(t\xi)\widehat{\rho_1}(-t\xi)
\widehat{\sigma_2}(t\eta)\widehat{{\sigma_2}}(-t\eta) \frac{dt}{t}.
\end{align}

Now multiply \eqref{tel-ftc} by $\widehat{\boldsymbol{F}}(\xi,-\xi,\eta,-\eta)$ and integrate in the variables $\xi,\eta$.   
It remains to use    \eqref{pf-tel} and to  evaluate the right hand-side of \eqref{tel-ftc}  as  $-|\widehat{\rho_1}(0)|^2|\widehat{\rho_2}(0)|^2$ times
\begin{align*}
 \int_{\R^2}\widehat{\boldsymbol{F}}(\xi,-\xi,\eta,-\eta)d\xi d\eta  & =   \int_{\R^2}\boldsymbol{F}(x,x,y,y) dxdy \\
& =   \int_{\R^2} F_1(x,y)F_2(x,y)F_3(x,y)F_4(x,y)dxdy.
\end{align*} 
This proves the claim.
\end{proof}

To apply  Lemma \ref{tel} we would like to have $\varphi^{(u)} =\varphi^{(-u)}$, as then we would get
 $${\Lambda}_{|\varphi^{(u)}|,|\varphi^{(-u)}|,\psi^{(v)},\psi^{(v)}}={\Lambda}_{|\varphi^{(u)}|,\psi^{(v)}}.$$
However, we do not have $\varphi^{(u)} =\varphi^{(-u)}$ in general. 
For this  and to circumvent possible lack of smoothness of $|\varphi^{(\pm u)}|$, we dominate  $|\varphi^{(\pm u)}|$ with a superposition of the Gaussian exponential functions. Consider 
\begin{align*}
\Phi(x):=\int_1^\infty \frac{1}{\alpha^5} e^{-\left( \frac{x}{\alpha}\right)^2} d\alpha =  \frac{1}{2x^4}({1-e^{-x^2}(x^2+1)}).
\end{align*}
The function $\Phi$  is positive, continuous at zero and for large $x$ comparable to $x^{-4}$. 
Let us denote the $\mathrm{L}^1$-normalized Gaussian rescaled by a parameter $\alpha >0$ by
\begin{align}\label{gauss}
g_{\alpha}(x)&:=\frac{1}{\sqrt{\pi}\alpha} e^{-\left( \frac{x}{\alpha}\right)^2}.
\end{align}
Then we can write  
$$\Phi=\pi^{-1/2} \int_1^\infty \frac{1}{\alpha^4}\, g_\alpha\, d\alpha.$$
Since $|\varphi^{(\pm u)}|$ satisfies the decay estimate \eqref{fct-est}, we can bound it pointwise by $\Phi$ multiplied by some positive constant which is uniform in $u$. 
 Positivity of the integrands in 
\begin{align}\label{pf-pos}\nonumber
{\Lambda}_{|\varphi^{(u)}|,|\varphi^{(-u)}|,\psi^{(v)},\psi^{(v)}}(F_1,F_2,F_2,F_1) &=  \int_0^\infty \int_{\R^4}  \Big ( \int_{\R} F_1(x,y)F_2(x',y)  [\psi^{(v)}]_{t}(q-y)  dy \Big )^2\\
 &  [|\varphi^{(u)}|]_t(p-x) \, [|\varphi^{(-u)}|]_t(p-x') dxdx'dpdq\frac{dt}{t}
\end{align}
then allows us to dominate           
\begin{align*}
{\Lambda}_{|\varphi^{(u)}|,|\varphi^{(-u)}|,\psi^{(v)},\psi^{(v)}}(F_1,F_2,F_2,F_1) \lesssim  \int_1^\infty\int_1^\infty \Lambda_{g_\alpha,g_\beta,{\psi^{(v)}},{\psi^{(v)}}}(F_1,F_2,F_2,F_1) 
\frac{d\alpha}{\alpha^4} \frac{d\beta}{\beta^4}.
\end{align*}
To reduce to only one scaling parameter in the last line we split the integration into the regions $\alpha\geq \beta$ and $\alpha < \beta$.
By symmetry it suffices to estimate the region $\alpha\geq \beta$ only, on which we bound 
$\beta g_\beta \leq \alpha g_\alpha$
for $\alpha,\beta \geq 1$. This 
leaves us with having to estimate
$$\int_1^\infty \Lambda_{g_\alpha,{\psi^{(v)}}}(F_1,F_2,F_2,F_1) \frac{d\alpha}{\alpha^3}.$$

We shall now apply Lemma \ref{tel} with $(\rho_1,\sigma_1)=(g_\alpha,h_\alpha)$ and $(\rho_2,\sigma_2)=(\phi,\psi^{(v)})$, where we define
$h_{\alpha}(x):={\alpha}(g_{\alpha})'(x)$ and 
$\phi$ is defined via
\begin{align}\label{root2}
\widehat{\phi}(\xi):=\left( \int_{\xi}^\infty | \widehat{\psi^{(v)}}(\tau) |^2 \frac{d\tau}{\tau} \right)^{1/2}.
\end{align}
Since $|\widehat{\psi^{(v)}}|^2 = \widehat{\vartheta_2}$,  by Lemma \ref{lemma:bump} the function $\widehat{\phi}$ belongs to $C_0^\infty(\R)$. 
Note that the two  pairs of functions $(\rho_i,\sigma_i)$ satisfy \eqref{tel-theid}, which  follows by a straightforward calculation.
Lemma \ref{tel} now
 yields
\begin{align} \label{after1tel} 
{\Lambda}_{g_\alpha,\psi^{(v)}}(F_1,F_2,F_2,F_1)   & =       - {\Lambda}_{h_\alpha,{\phi}}(F_1,F_2,F_2,F_1) +  \widehat{\phi}(0)^2 \int_{\R^2} F_1^2 F_2^2. \end{align}
By the Cauchy-Schwarz inequality we have 
$$\int_{\R^2}F_1^2F_2^2\leq \|F_1\|^2_{\mathrm{L}^4(\R^2)}\|F_2\|^2_{\mathrm{L}^4(\R^2)} = 1,$$ 
so it remains to consider the first term on the right hand-side of  \eqref{after1tel}. 

To estimate it we repeat the just performed steps, which   will further separate  the functions $F_1,F_2$. 
The role of $\varphi^{(\pm u)}$ is now taken over by $\phi$ and the role of $\psi^{(\pm v)}$ is assumed by $h_\alpha$. 
Therefore we can group the integrals in ${\Lambda}_{h_\alpha,\phi}$ according to the variables $x,x'$, and bound
$|{\Lambda}_{h_\alpha,\phi}(F_1,F_2,F_2,F_1)| $ by
\begin{align*}
  \int_0^\infty \int_{\R^4}  \left | \int_{\R}F_1(x,y)F_1(x,y')[h_{\alpha}]_t(p-x)  dx \right|
 \left|\int_{\R} F_2(x',y')F_2(x',y) [h_{\alpha}]_t(p-x') dx' \right| & \\
 [|\phi|]_{t}(q-y)[|\phi|]_{t}(q-y') dydy'dpdq\frac{dt}{t}. &
\end{align*} 
Applying the Cauchy-Schwarz inequality   we obtain
\begin{align*}
|{\Lambda}_{h_\alpha,{\phi}}(F_1,F_2,F_2,F_1)| \leq {\Lambda}_{h_\alpha,|\phi|}(F_1,F_1,F_1,F_1)^{1/2}
{\Lambda}_{{h}_\alpha,|\phi|}(F_2,F_2,F_2,F_2)^{1/2}.
\end{align*}

Now we dominate the rapidly decaying function $|\phi|$ by a positive constant times $\Phi$, which gives for the first factor
\begin{align}\label{f2}
{\Lambda}_{h_\alpha,|\phi|}(F_1,F_1,F_1,F_1) \lesssim 
\int_1^\infty \int_1^\infty \Lambda_{h_\alpha,{h}_\alpha,g_\gamma,g_\delta}(F_1,F_1,F_1,F_1)\frac{d\gamma}{\gamma^4} \frac{d\delta}{\delta^4}.
\end{align}
By symmetry it again suffices to estimate
$$\int_1^\infty {\Lambda}_{h_\alpha,g_\gamma}(F_1,F_1,F_1,F_1) \frac{d\gamma}{\gamma^3}.$$
Lemma \ref{tel} with $(\rho_1,\sigma_1)=(g_\alpha,h_\alpha)$ and $(\rho_2,\sigma_2)=(g_\gamma,h_\gamma)$ gives
\begin{align*}%\label{pf-last}
{\Lambda}_{h_\alpha,g_\gamma}(F_1,F_1,F_1,F_1) & =    - {\Lambda}_{g_\alpha,h_\gamma}(F_1,F_1,F_1,F_1) +  \int_{\R^2}F_1^4 .
\end{align*}

The key gain we obtain from having reduced to a single function $F_1$ is that 
\begin{align}\label{last-form}
 {\Lambda}_{g_\alpha,h_\gamma}(F_1,F_1,F_1,F_1)\geq 0, 
\end{align}
which can be seen by writing the form in \eqref{last-form} in an analogous way as in \eqref{pf-pos} and using positivity of $g_\alpha$. By our normalization, $\int_{\R^2}F_1^4 = 1$.
Thus, 
$${\Lambda}_{h_\alpha,g_\gamma}(F_1,F_1,F_1,F_1) \leq 1. $$ 
This establishes the desired estimate for   $\widetilde{\Lambda}_{\varphi^{(u)},\varphi^{(-u)},\psi^{(v)},\psi^{(-v)}}$.

\section{Appendix}
In this appendix we give the following remaining proof.
\begin{proof}[Proof of Lemma \ref{lemma:bump}] 
We construct a function $f$ which has the prescribed behavior near the endpoints of its support, so that the considered square roots are evidently smooth.
The construction  essentially  consists of algebraic manipulations of   $\varphi(x):=e^{-\frac{1}{x}}\mathbf{1}_{(0,\infty)}(x)$. 

Consider the function
\begin{align*}
g(x):=c\, \varphi_1(x)\varphi_2(2-\frac{2}{\varepsilon}x),
\end{align*}
where $\varphi_1$ and $\varphi_2$ are defined as
\begin{align*}
\varphi_1(x):= ((3-x)\varphi'(x))'\hspace{1cm}\mathrm{and}\hspace{1cm} \varphi_2(x):=\frac{\varphi(x)}{\varphi(x)+\varphi(1-x)}.
\end{align*}
The constant $c>0$ is chosen such that $\int_\R g =1$.
 The function $g$ is smooth, non-negative and supported on $[0,\varepsilon]$. Since $\varphi_2$ equals $1$ for $x\geq 1$, for $\delta:=\varepsilon/2$ we have
 \begin{align*}
 g  = c\, \varphi_1 \;\; \mathrm{on} \;\; (-\infty,\delta).
 \end{align*}
  The factor $(3-x)$ in the definition of $\varphi_1$   will be convenient when investigating \eqref{square-root}.
  
We consider the  antiderivative $$f(x):=\int_{-\infty}^x g(t-1)-g(3-t) dt,$$
which is smooth and even about $x=2$, i.e. $f(x)=f(4-x)$. Moreover, it is supported on $[1,3]$,  positive on $(1,3)$ and   constantly equals $1$ on $[1+\varepsilon,3-\varepsilon]$. We have
\begin{align}\label{eqn:f}
f(x)=c\, (4-x)\varphi'(x-1)\;\;\mathrm{on}\;\;(-\infty,1+\delta).
\end{align}
Thus, $f^{1/2}$ is smooth at $x=1$. Smoothness at $x=3$ follows by symmetry. 

Consider the integral in \eqref{square-root}, 
which is due to oddness of the integrand equal to 
\begin{align*}
h(x):=\int_{-\infty}^x - \frac{f(t)+f(-t)}{t} dt.
\end{align*}
The function $h$ is even, supported on $[-3,3]$ and  positive on $(-3,3)$. Using $f(-t)=f(t+4)$ and \eqref{eqn:f} we see that
\begin{align*}
h(x) = c\, \varphi(x+3)  \;\;\mathrm{on}\;\;(-\infty,-3+\delta).
\end{align*}
This shows smoothness of $h^{1/2}$ at $x=-3$.
By symmetry the same holds at $x=3$, which establishes the claim of the lemma.
\end{proof}

%%%%%%%%%%%%%%%%%%%%%%%%%%%%%%%%%%%%%%%%%%%%%%%%%%%%%%%%%%%%%%%%%%%%%%%%%%%%%%%%%%%%%%%%%%%%%%%%%%%%%%%%%%%%%%%%%

\end{document}